\documentclass[a4paper,10pt]{article}
\usepackage{amscd}
\usepackage{amsmath}
\usepackage{bbding}
\usepackage{amsfonts}
\usepackage{cite}
\usepackage{mathrsfs}
\usepackage{amssymb}
\allowdisplaybreaks[2]
\usepackage{leftidx}
\usepackage{graphicx,latexsym,amsmath,amssymb,amsthm,enumerate}
\usepackage{float}

\usepackage[dvipdfm,pdfstartview=FitH,
           colorlinks=true,
             pdfborder=001,
           citecolor=blue]{hyperref}

9.5in\textwidth 6.4in \hoffset -2.0cm \DeclareFontEncoding{OT2}{}{}
\DeclareTextSymbol{\cyrsftsn}{OT2}{126}
\DeclareTextSymbol{\textnumero}{OT2}{125}

\theoremstyle{definition}
\newtheorem{theorem}{Theorem}[section]
\newtheorem{lemma}{Lemma}[section]
\newtheorem{corollary}{Corollary}[section]
\newtheorem{proposition}{Proposition}[section]
\newtheorem{definition}{Definition}[section]
\newtheorem{remark}{Remark}[section]
\newtheorem{example}{Example}[section]

\begin{document}
\title{{\LARGE\bf{Approximate solutions in multiobjective interval-valued optimization problems: Existence theorems and optimality conditions}}\thanks{This work was supported by the National Natural Science Foundation of China (12201251,12171339), the Natural Science Foundation of Sichuan Province (2022NSFSC1830) and Southwest Minzu University Research Startup Funds (RQD2022035).}}
\author{Chuang-liang Zhang $^{a}$\thanks{E-mail: clzhang1992@sina.com (C.L. Zhang)}\;, Yun-cheng Liu  $^{b}$\thanks{E-mail:
yuncheng.liu@swun.edu.cn (Y.C. Liu)}\;,
 Nan-jing Huang $^{c}$\thanks{Corresponding author, E-mail: nanjinghuang@hotmail.com; njhuang@scu.edu.cn (N.J. Huang)}
\\
{\small\it $^a$ School of Mathematics, Jiaying University, Meizhou 514015, Guangdong, P.R. China}\\
{\small\it $^b$ School of Mathematics, Southwest Minzu University, Chengdu 610041, Sichuan, P.R. China}\\
{\small\it $^c$ Department of Mathematics, Sichuan University, Chengdu 610064, Sichuan, P.R. China}}
\date{ }
\maketitle
\begin{flushleft}
\hrulefill\\
\end{flushleft}
 {\bf Abstract}.
This paper is devoted to the study of approximate solutions for a multiobjective interval-valued optimization problem based on an interval order. We establish new existence theorems of approximate solutions for such a problem under some mild conditions. Moreover, we give KKT optimality conditions for approximate solutions for such a problem whose associated functions are nonsmooth and nonconvex. We also propose the approximate KKT optimality condition of an approximate solution for such a problem. Finally, we apply some obtained results to a noncooperative game involving the multiobjective interval-valued function.
\begin{flushleft}
\hrulefill
\end{flushleft}

\section{Introduction}\label{sec1}

Multiobjective optimization has been widely used to deal with many problems stemmed from engineering, games theory, and biology (see  \cite{Jahn,Solei}). We notice that numerical algorithms for solving the multiobjective optimization problems (for short, MOPs) are related to approximate solutions. Approximate solutions of MOPs can be considered as feasible points whose objective values are near to optimal values. About studying approximate solutions for MOPs, there are numerous approaches. It is worth mentioning that Loridan \cite{Loridan} introduced and studied $\varepsilon$-quasi minimal solutions for MOPs based on Ekeland's variational principle \cite{Ekeland1} (for short, EVP).
In past decades, various types of approximate solutions of MOPs have been investigated by many authors (see \cite{CK,Gao,GM,Guti,Huerga,Son}).

In many real-life situations, the multiobjective decision-making model is often embedded with uncertain parameters that makes it difficult to obtain precise date. As is well known, interval-valued methods provide a powerful tool to deal with optimization models involving uncertain data.  There have been numerous studies addressing the interval-valued
optimization/analysis theory with applications (see \cite{BS,Moore,SB,SJ,Wu}). Recently, Zhang and Huang \cite{ZH} extended the classical results for EVP to interval cases and provided some applications; Zhang et al. \cite{ZHO} presented interval-valued variational methods by using some results of \cite{ZH}, and gave some theoretical applications concerning interval-valued functions; Villanueva et al. \cite{ Vill} obtained some optimality conditions for interval-valued optimization problems (for short, IOPs) based on the gH-directional derivative. Very recently, Sahu et al. \cite{Sahu} investigated efficient solutions for a class of vector optimization problems on an extended interval vector space and provided an application to portfolio optimization.

It is worth mentioning that there has been increasing and high demanding for investigating multiobjective interval-valued optimization problems (for short, MIOPs), from theoretical, numerical and practical aspects, see, for example \cite{COC,Hung,IT,OHCR,Singh,SAS,SGA,UPZ,Wu,Wu1}, and the references therein. Recently, Tuyen \cite{Tuyen} studied some existence results and KKT optimality conditions for approximate solutions of IOPs involving convexity; Zhang and Huang \cite{ZH} obtained the existence of approximate solutions for IOPs by using interval EVP. However, to the best of our knowledge, existence theorems and optimality conditions for approximate solutions of MIOPs have not been considered in the literature, apart from \cite{Hung} wherein some KKT-type optimality conditions were given for approximate quasi Pareto solutions of MIOPs by employing the Mordukhovich/limiting subdifferential. Thus, it would be important to investigate existence theorems and optimality conditions for approximate solutions of MIOPs under some mild conditions.

The current paper is devoted to the study of approximate solutions of MIOPs. The main contributions of this work focus on some points as following: (i) a new multiobjective interval-valued version of EVP is established;  (ii) new existence theorems for approximate solutions of MIOP are given; (iii) by using the nonsmooth analysis technique, new results on
KKT optimality conditions for approximate solutions of MIOP are proposed under nonconvex and nonsmooth conditions. Moreover, modified $\epsilon$-KKT points for the weak minimal and approximate solution of such problems are presented; (iv) new results on KKT optimality conditions for approximate
Nash equilibrium of a game with the multiobjective interval-value case are derived.

The rest of this paper is organized as follows. Section \ref{sec2} contains some basic definitions and preliminary results. Section \ref{sec3} shows new results of existence of approximate solutions for MIOP. Section \ref{sec4} presents KKT optimality conditions and the approximate KKT condition for approximate solutions of MIOP. In Section \ref{sec5}, we present an application to a noncooperative game for multiobjective interval-valued functions. To finish, Section \ref{sec6} gives some conclusions and our future research works.

\section{Preliminaries}\label{sec2}
\setcounter{equation}{0}

Let $\mathcal{I}$ be the set of all closed and bounded intervals in $\mathbb{R}$, namely, $\mathcal{I}=\{[\underline{q},\overline{q}]:\underline{q}\leq \overline{q}, \underline{q},\overline{q}\in\mathbb{R}\}$.
Let $\mathbb{R}^m_+$ be the nonnegative orthant of $\mathbb{R}^m$.
If $Q\subset \mathbb{R}^n$, then $\textrm{co} Q$ (resp., $\textrm{cl} Q$) is the convex hull (resp., the closure) of $Q$.

Letting $Q=[\underline{q},\overline{q}],R=[\underline{r},\overline{r}]\in\mathcal{I}$ with $\alpha\in\mathbb{R}$, we define
\begin{equation}
Q+R=[\underline{q}+\underline{r},\overline{q}+\overline{r}]\nonumber\\
\end{equation}
and
\begin{equation}
\alpha Q=
\begin{cases}
[\alpha \underline{q},\alpha \overline{q}] & \text{if $\alpha\geq0,$}\\
[\alpha \overline{q},\alpha \underline{q}] &\text{if $\alpha<0.$}\nonumber
\end{cases}
\end{equation}

For $Q=[\underline{q},\overline{q}]\in\mathcal{I}$, the center and width are defined as
$q^c=\frac{\underline{q}+\overline{q}}{2}$ and $q^w=\frac{\overline{q}-\underline{q}}{2}$.
Clearly, $\underline{q}=q^c-q^w$ and $\overline{q}=q^c+q^w$.

Also, the norm $\|\cdot\|_{\mathcal{I}}$ on $\mathcal{I}$ can be defined as \cite{Moore}:
$\|Q\|_{\mathcal{I}}=\max\{|\underline{q}|,|\overline{q}|\}$.

Recall the notion of generalized Hukuhara difference (for short, gH) on $\mathcal{I}$ as follows.
\begin{definition}(\cite{BS,SB})\label{de2.1}
Let $Q,R \in \mathcal{I}$. The gH between $Q$ and $R$ is defined as
$$
Q\ominus_{gH}R=P\Leftrightarrow Q=R+P\; \textrm{or}\; R=Q-P,
$$
where $P\in\mathcal{I}$.
\end{definition}
Obviously, if $Q=[\underline{q},\overline{q}]$ and $R=[\underline{r},\overline{r}]$, then $Q\ominus_{gH}R=\left[\min\{\underline{q}-\underline{r},\overline{q}-\overline{r}\},\max\{\underline{q}
-\underline{r},\overline{q}-\overline{r}\}\right]$.

The Hausdorff distance $h$ on $\mathcal{I}$ is given by
$$h(Q,R)=\max\left\{\sup_{q\in Q}d_R(q),\sup_{r\in R}d_Q(r)\right\},\quad \forall Q,R \in \mathcal{I},$$
where $d_R(q):=\inf_{r\in R}|q-r|$.
Furthermore, if $Q=[\underline{q},\overline{q}]$ and $R=[\underline{r},\overline{r}]$, then
$$h(Q,R)=\|Q\ominus_{gH}R\|_{\mathcal{I}}=\max\{|\underline{q}-\underline{r}|,|\overline{q}-\overline{r}|\}.$$
Clearly, $(\mathcal{I},h)$ is a complete metric space.

A sequence $\{Q_n\}\subset\mathcal{I}$ is said to be convergence to $Q\in\mathcal{I}$  if $h(Q_n,Q)\rightarrow0$. Obviously, if $Q_n=[\underline{q}_n,\overline{q}_n]$ and $Q=[\underline{q},\overline{q}]$, then
$h(Q_n,Q)\rightarrow 0$ iff $|\underline{q}_n-\underline{q}|\rightarrow0$ and $|\overline{q}_n-\overline{q}|\rightarrow0$.

We introduce two interval orders on $\mathcal{I}$ as follows.
\begin{definition}\label{de2.2}(\cite{IT})
Let $Q=[\underline{q},\overline{q}]$ and $R=[\underline{r},\overline{r}]$. The interval orders $\preccurlyeq_{CW}$ and $\prec_{CW}$ on $\mathcal{I}$ are defined by
$$Q\preccurlyeq_{CW} R\Leftrightarrow  q^c\leq r^c\; \textrm{and} \;q^w\leq r^w;\quad
Q\prec_{CW} R \Leftrightarrow \ q^c< r^c\; \textrm{and} \;q^w< r^w.$$
\end{definition}
\begin{remark}\label{re2.1}
Clearly, (i) $\preccurlyeq_{CW}$ is the partial order, i.e., it is reflexive, antisymmetric and transitive;

(ii) for each $P\in \mathcal{I}$, $Q\preccurlyeq_{CW}(\textrm{resp}.,\prec_{CW})~  R\Leftrightarrow Q+P\preccurlyeq_{CW}(\textrm{resp}., \prec_{CW})~  R +P$;

(iii) if  $Q_1\preccurlyeq_{CW}(\textrm{resp}.,\prec_{CW})~  R_1$ and  $Q_2\preccurlyeq_{CW}(\textrm{resp}.,\prec_{CW})~  R_2$ for $Q_1,Q_2,R_1,R_2\in\mathcal{I}$, then $Q_1+Q_2\preccurlyeq_{CW}(\textrm{resp}.,\prec_{CW})~  R_1+R_2$;

(iv) for each $\alpha\geq0$, $\alpha Q\preccurlyeq_{CW} \alpha R$.
\end{remark}

Let $f:\mathbb{R}^n\to\mathcal{I}$ be an interval-valued function (for short, IVF) denoted by $f(u)=[\underline{f}(u),\overline{f}(u)]$ with $\underline{f}(u)\leq \overline{f}(u)$.
In this case, we write $f^c(u)=\frac{\underline{f}(u)+\overline{f}(u)}{2}$ and $f^w(u)=\frac{\overline{f}(u)-\underline{f}(u)}{2}$.

Recall some concepts of IVFs and their properties as follows.
\begin{definition}\label{de2.3}
An IVF $f:\mathbb{R}^n\to\mathcal{I}$ is said to be
\begin{itemize}
\item [(i)] $\preccurlyeq_{CW}$-lsc ( $\preccurlyeq_{CW}$-lower semicontinuous) on $\mathbb{R}^n$ if $f^c$ and $f^w$ are lsc on $\mathbb{R}^n$;

\item [(ii)] \cite{Wu} continuous on $\mathbb{R}^n$ if for every $u_0\in\mathbb{R}^n$ and each $\{u_n\}\subset \mathbb{R}^n$ satisfying $u_n\rightarrow u_0$, one has
    $$\lim_{n\rightarrow +\infty}h(f(u_n),f(u_0))=0;$$

\item[(iii)]  locally Lipschitz continuous on $\mathbb{R}^n$ if for every $u_0\in \mathbb{R}^n$, there exist a neighbourhood $U(u_0)$ and some constant $L>0$ satisfying
    $$\|f(u_1)\ominus_{gH}f(u_2)\|_{\mathcal{I}}\leq L\|u_1-u_2\|,\;\forall u_1,u_2\in U(u_0),$$
    where $\|\cdot\|$ denotes the norm on $\mathbb{R}^n$.
\end{itemize}
\end{definition}

\begin{remark}\label{re2.2}
\begin{itemize}
\item [(a)] Clearly, if $f$ and $g$ are $\preccurlyeq_{CW}$-lsc on $\mathbb{R}^n$, then $f+g$ is also $\preccurlyeq_{CW}$-lsc.

\item [(b)] If $f$ is $\preccurlyeq_{CW}$-lsc on $\mathbb{R}^n$, then for each $Q\in\mathcal{I}$, $\{u\in\mathbb{R}^n:f(u)\preccurlyeq_{CW} Q\}$ is closed.

\item [(c)] By \cite{SGA}, it follows that $f$ is continuous iff $f^c$ and $f^w$ are continuous.
Moreover, if $f$ is continuous, then it is  $\preccurlyeq_{CW}$-lsc.

\item[(d)] The local Lipschitz continuity of IVFs implies the continuity. Furthermore,
$f$ is locally Lipschitz continuous iff $\underline{f}$ and $\overline{f}$ are locally Lipschitz continuous. In this case, $f^c$ and $f^w$ are also
 locally Lipschitz continuous.
\end{itemize}
\end{remark}

Consider the Clarke directional derivative of the local Lipschitz continuity of function $\psi:\mathbb{R}^n\to\mathbb{R}$ at $z$ in the direction $w$ as
$$\psi^0(z;w):=\limsup_{u\rightarrow z,t\rightarrow0^+}\frac{\psi(u+tw)-\psi(u)}{t}.$$
The generalized gradient of $\psi$ at $z$ can be defined as
$$
\partial^0\psi(z):=\{z^*\in \mathbb{R}^n:\psi^0(z;w)\geq\langle z^*,w\rangle,\forall w\in \mathbb{R}^n\}.
$$
From \cite{Clarke2,Clarke3,Clarke4}, we know that
\begin{itemize}
\item[(i)] $\partial^0(\lambda_1\psi_1(z)+\lambda_2\psi_2(z))\subset\lambda_1\partial^0\psi_1(z)+\lambda_2\partial^0\psi_2(z)$, where $\psi_1,\psi_2$ are real-valued functions on $\mathbb{R}^n$, $\lambda_1,\lambda_2\in\mathbb{R}$;

\item[(ii)] $\partial^0\phi(z)\subset co\{\partial^0\psi_i(z):\phi(z)=\psi_i(z),i=1,2,\cdots,m\}$, where $\phi(z):=\max\{\psi_i(z):i=1,2,\cdots,m\}$.
\end{itemize}

For $\emptyset\not=A\subset\mathbb{R}^n$,  the Clarke tangent cone to $A$ at $z$ is defined by
$$
T_A(z):=\{w\in\mathbb{R}^n:d_A^0(z;w)=0\},
$$
where $d_A^0(z;w)=\limsup_{u\rightarrow z,t\rightarrow0^+}\frac{d_A(u+tw)-d_A(u)}{t}$,
and the Clarke normal cone to $A$ at $z$ is defined as
$$
N_A(z):=\{z^*\in\mathbb{R}^n:\langle z^*,u\rangle\leq0,\forall u\in T_A(z)\}.
$$

Now we give the notion of weakly generalized gradient of IVFs.
Let $f:\mathbb{R}^n\to \mathcal{I}$ be a locally Lipschitz continuous IVF. Then the weakly generalized gradient of $f$ can be defined as follows
$$\partial^0 f(u):=\textrm{co}\{\partial^0 f^c(u),\partial^0 f^w(u)\}.$$

\begin{remark}\label{re2.3}
Obviously, for $\overline{u}\in\mathbb{R}^n$ and $\delta>0$, one has
$\partial^0 (f(\cdot)+[0,\delta\|\cdot-\overline{u}\|])\subset \partial^0f(\cdot) +\delta\mathbb{B}^n,$ where $\mathbb{B}^n$ denotes the closed unit ball in $\mathbb{R}^n$.
\end{remark}

In this paper, we consider the following multiobjective interval-valued optimization problem (MIOP):
$$
\quad \textrm{Min}\; F(u):=(F_1(u),F_2(u),\cdots,F_m(u))\quad \textrm{s.t.} \;\; u\in S,
$$
where $g_j: \mathbb{R}^n\to\mathbb{R}$ for $j\in J:=\{1,2,\cdots,p\}$, $S:=\{u\in \mathbb{R}^n:g_j(u)\leq0,\forall j\in J\}$ and $F_k: \mathbb{R}^n\to \mathcal{I}$ is an IVF with $F_k(u):=[\underline{F}_k(u),\overline{F}_k(u)]$, $F^c_k(u)=\frac{\underline{F}_k(u)+\overline{F}_k(u)}{2}, F^w_k(u)=\frac{\overline{F}_k(u)-\underline{F}_k(u)}{2}$ for $k\in K:=\{1,2,\cdots,m\}$.

Now we introduce the following notions of approximate solutions of MIOPs.
\begin{definition}\label{de2.5}
Let $U(z_0)$ be a neighbourhood of $z_0\in S$. We say that $z_0$ is
\begin{itemize}
\item [(i)] a locally weak minimal solution for MIOP if there is no $z\in S\cap U(z_0)$ satisfying $F_i(z)\prec_{CW}F_i(z_0)$ for all $i\in K$;
\item[(ii)] a locally weak $\varepsilon$-minimal solution for MIOP if for any given $\varepsilon:=(\varepsilon_1,\varepsilon_2,\cdots,\varepsilon_m)\in\mathbb{R}^m_+$,
 there is no $z\in S\cap U(z_0)$ satisfying $F_i(z)+[0,\varepsilon_i]\prec_{CW}F_i(z_0)$ for all $i\in K$;
 \item[(iii)] a locally weak $\varepsilon$-quasi-minimal solution for MIOP if for any given $\varepsilon:=(\varepsilon_1,\varepsilon_2,\cdots,\varepsilon_m)\in\mathbb{R}^m_+$,
there is no $z\in S\cap U(z_0)$ satisfying $F_i(z)+[0,\varepsilon_i\|z-z_0\|]\prec_{CW}F_i(z_0)$ for all $i\in K$.
\end{itemize}
By simply saying $U(z_0)=S$, we say that the above solutions are global on $S$. In this cases, the weak minimal solution set,  the weak $\varepsilon$-minimal solution set, and the weak $\varepsilon$-quasi-minimal solution set of MIOP are denoted by $M(F,S)$, $M(F,S,\varepsilon)$ and $QM(F,S,\varepsilon)$, respectively.
\end{definition}

\begin{remark}\label{re2.4} If $\varepsilon=0$, then the notions of a weak $\varepsilon$-minimal solution and a weak $\varepsilon$-quasi-minimal solution coincide with the weak minimal solution.
\end{remark}

\begin{proposition}\label{p2.1}
Let $\varepsilon_0>0$ and $\varepsilon':=(\sqrt{\varepsilon_0},\sqrt{\varepsilon_0},\cdots,\sqrt{\varepsilon_0})\in\mathbb{R}^m_+$. If $\overline{u}\in QM(F,S,\varepsilon')$, then $\overline{u}\in M(F,\overline{B}(\overline{u},\sqrt{\varepsilon_0}),\varepsilon)$, where
$\overline{B}(\overline{u},\sqrt{\varepsilon_0}):=\{z\in S: \|z-\overline{u}\|\leq\sqrt{\varepsilon_0}\}$ and $\varepsilon:=(\varepsilon_0,\varepsilon_0,\cdots,\varepsilon_0)\in\mathbb{R}^m_+$.
\end{proposition}
\begin{proof}
Let $\overline{u}\in QM(F,S,\varepsilon')$. If $\overline{u}\not\in  M(F,\overline{B}(\overline{u},\sqrt{\varepsilon_0}),\varepsilon)$, then
there is $z'\in \overline{B}(\overline{u},\sqrt{\varepsilon_0})$ satisfying $F_i(z')+[0,\varepsilon_0]\prec_{CW}F_i(\overline{u})$ for all $i\in K$.
Then we have $F^c_i(z')+\frac{\varepsilon_0}{2}<F^c_i(\overline{u})$ and $F^w_i(z')+\frac{\varepsilon_0}{2}<F^w_i(\overline{u})$,
and so
$$F^c_i(z')+\frac{\sqrt{\varepsilon_0}\|z'-\overline{u}\|}{2}<F^c_i(\overline{u})-
\frac{\varepsilon_0}{2}+\frac{\sqrt{\varepsilon_0}\|z'-\overline{u}\|}{2}<F^c_i(\overline{u})-
\frac{\varepsilon_0}{2}+\frac{\sqrt{\varepsilon_0}\cdot\sqrt{\varepsilon_0}}{2}=F^c_i(\overline{u}),$$
$$F^w_i(z')+\frac{\sqrt{\varepsilon_0}\|z'-\overline{u}\|}{2}<F^w_i(\overline{u})-
\frac{\varepsilon_0}{2}+\frac{\sqrt{\varepsilon_0}\|z'-\overline{u}\|}{2}<F^c_i(\overline{u})-
\frac{\varepsilon_0}{2}+\frac{\sqrt{\varepsilon_0}\cdot\sqrt{\varepsilon_0}}{2}=F^w_i(\overline{u}).$$
This means that $F_i(z')+[0,\sqrt{\varepsilon_0}\|z'-\overline{u}\|]\prec_{CW}F_i(\overline{u})$ for all $i\in K$, which contradicts with $\overline{u}\in QM(F,S,\varepsilon')$. Thus, $\overline{u}\in  M(F,\overline{B}(\overline{u},\sqrt{\varepsilon_0}),\varepsilon)$
\end{proof}

To obtain multiobjective interval-valued versions of EVP, we need the following lemma.
\begin{lemma}\label{le2.1}(\cite{DHM})
Let $(Z,d)$ be a complete metric space and $T:Z\rightarrow 2^Z$ be a set-valued mapping. Suppose that
\begin{itemize}
\item[(i)] for any $z\in Z$, $z\in T(z)$, and $T(z)$ is closed;
\item[(ii)] for any $y\in T(z)$, $T(y)\subseteq T(z)$;
\item[(iii)] for any sequence $\{u_n\}$ in $Z$ with $u_{n+1}\in T(u_n)$ for all $n=1,2,\cdots$, one has $d(u_n,u_{n+1})\rightarrow0$.
\end{itemize}
Then there exists $\widehat{u}\in Z$ such that $T(\widehat{u})=\{\widehat{u}\}$.
\end{lemma}

\section{Existence theorems}\label{sec3}
\setcounter{equation}{0}
In this section, we will prove some new existence theorems for approximate solutions of MIOP. We first give the sufficient condition for the existence of weak $\varepsilon$-minimal solutions of MIOP.
\begin{theorem}\label{th3.1}
Let $\varepsilon\in\mathbb{R}^m_+\backslash\{0\}$. Then $M(F,S,\varepsilon)\neq\emptyset$.
\end{theorem}
\begin{proof}
Set $A(u)=\{z\in S: \sum_{i=1}^{m} F_i(z)+[0,\sum_{i=1}^m\varepsilon_i]\prec_{CW}\sum_{i=1}^{m} F_i(u)\}$ for $u\in S$. Fix $\overline{u}\in S$. If $A(\overline{u})=\emptyset$, then $\overline{u}\in M(F,S,\varepsilon)$.
 Suppose that $\overline{u}\not\in M(F,S,\varepsilon)$. Then there is $z'\in S$ satisfying $F_i(z')+[0,\varepsilon_i]\prec_{CW}F_i(\overline{u})$ for all $i\in K$. By Remark \ref{re2.1}, one has $\sum_{i=1}^{m}F_i(z')+[0,\sum_{i=1}^{m}\varepsilon_i]\prec_{CW}\sum_{i=1}^{m}F_i(\overline{u})$,
which contradicts with $A(\overline{u})=\emptyset$. Thus, if $A(\overline{u})=\emptyset$, then the proof is finshied. Assume that $A(\overline{u})\not=\emptyset$.
Then there is $u_1\in A(\overline{u})$ with $A(u_1)\neq\emptyset$ (otherwise, the proof is finished), and so $\sum_{i=1}^{m} F_i(u_1)+[0,\sum_{i=1}^m\varepsilon_i]\prec_{CW}\sum_{i=1}^{m} F_i(\overline{u})$. This means
that $\sum_{i=1}^{m} F^c_i(u_1)+\frac{1}{2}\sum_{i=1}^m\varepsilon_i<\sum_{i=1}^{m} F^c_i(\overline{u})$
and $\sum_{i=1}^{m} F^w_i(u_1)+\frac{1}{2}\sum_{i=1}^m\varepsilon_i<\sum_{i=1}^{m} F^w_i(\overline{u})$.
Repeated the above process, we can find $u_{n+1}\in A(u_n)$ with $A(u_n)\neq\emptyset$, $n=1,2,\cdots$, which implies that
 $\sum_{i=1}^{m} F_i(u_{n+1})+[0,\sum_{i=1}^m\varepsilon_i]\prec_{CW}\sum_{i=1}^{m} F_i(u_n)$.
Thus,  $\sum_{i=1}^{m} F^c_i(u_{n+1})+\frac{1}{2}\sum_{i=1}^m\varepsilon_i<\sum_{i=1}^{m} F^c_i(u_n)$
and $\sum_{i=1}^{m} F^w_i(u_{n+1})+\frac{1}{2}\sum_{i=1}^m\varepsilon_i<\sum_{i=1}^{m} F^w_i(u_n)$.
We can deduce that
$$
\sum_{i=1}^{m} F^c_i(u_{n+1})+\frac{n+1}{2}\sum_{i=1}^m\varepsilon_i<\sum_{i=1}^{m} F^c_i(\overline{u}),\;
\sum_{i=1}^{m} F^w_i(u_{n+1})+\frac{n+1}{2}\sum_{i=1}^m\varepsilon_i<\sum_{i=1}^{m} F^w_i(\overline{u}).
$$
By $F^w_i(u)\geq0$ for $i\in K$, we have $\frac{1}{2}\sum_{i=1}^m\varepsilon_i<\frac{1}{n+1}\sum_{i=1}^{m} F^w_i(\overline{u})$. Thus, by $\varepsilon\in\mathbb{R}^m_+\backslash\{0\}$, when $n\to +\infty$, we get $0<\sum_{i=1}^{m}\varepsilon_i\leq0$, which is a contradiction.
\end{proof}

Next we establish some multiobjective interval-valued versions of EVP.
\begin{proposition}\label{p3.1}
Let $X$ be a real Banach space, $\|\cdot\|_X$ be the norm of $X$, and $F_k:X\to \mathcal{I}$ be $\preccurlyeq_{CW}$-lsc for $k\in K$.
Then for $\varepsilon\in\mathbb{R}^m_+\backslash\{0\}$ and $x_0\in X$ with $\sum_{k=1}^{m}F_k(u)+[0,\sum_{k=1}^{m}\varepsilon_k]\not\prec_{CW}\sum_{k=1}^{m}F_k(x_0)$ for all $u\in X$, there is a point $\overline{u}\in X$ satisfying
\begin{itemize}
\item[(a)]  $\sum_{k=1}^{m}F_k(u)+[0,\sum_{k=1}^{m}\varepsilon_k]\not\prec_{CW}\sum_{k=1}^{m}F_k(\overline{u}),\;\forall u\in X$;
\item[(b)] $\|x_0-\overline{u}\|_X\leq\frac{\sum_{k=1}^{m}\varepsilon_k}{\sum_{k=1}^{m}\sqrt{\varepsilon_k}}$;
\item[(c)] $\sum_{k=1}^{m}F_k(u)+[0, \sum_{k=1}^{m}\sqrt{\varepsilon_k}\|u-\overline{u}\|_X]\not\prec_{CW}\sum_{k=1}^{m}F_k(\overline{u})$, $\forall u\in X$.
\end{itemize}
\end{proposition}
\begin{proof}
For $\varepsilon=(\varepsilon_1,\varepsilon_2,\cdots,\varepsilon_m)\in\mathbb{R}^m_+\backslash\{0\}$ and $x_0\in X$, let
$$T_0:=\left\{u\in X: \sum_{k=1}^{m}F_k(u)+[0, \sum_{k=1}^{m}\sqrt{\varepsilon_k}\|u-x_0\|_X]\preccurlyeq_{CW}\sum_{k=1}^{m}F_k(x_0)\right\}.$$
Clearly, by Remark \ref{re2.1}, one has $x_0\in T_0$. It follows from Remark \ref{re2.2} that
$$
\sum_{k=1}^{m}F_k(\cdot)+[0,\sum_{k=1}^{m}\sqrt{\varepsilon_k}\|\cdot-x_0\|_X]
$$
is $\preccurlyeq_{CW}$-lsc and thus $T_0$ is closed.

If $T_0=\{x_0\}$, then we take $\overline{u}:=x_0$, this proof is complete.

Otherwise, $T_0\neq\{x_0\}$. Then we can define a mapping $T: T_0\rightarrow 2^{T_0}$ as
$$
T(u):=\left\{v\in T_0: \sum_{k=1}^{m}F_k(v)+[0, \sum_{k=1}^{m}\sqrt{\varepsilon_k}\|v-u\|_X]\preccurlyeq_{CW}\sum_{k=1}^{m}F_k(u)\right\},\quad\forall u\in T_0.
$$
Now we check that $T$ satisfies the conditions (i)-(iii) of Lemma \ref{le2.1}.  Obviously, for every $u\in T_0$, $u\in T(u)$, and $T(u)$ is closed. Moreover, we claim that, for any $u\in T_0$, $v\in T(u)$, $T(v)\subseteq T(u)$. Indeed, for $v\in T(u)$ and $w\in T(v)$, we have
$$\sum_{k=1}^{m}F_k(v)+[0, \sum_{k=1}^{m}\sqrt{\varepsilon_k}\|v-u\|_X]\preccurlyeq_{CW}\sum_{k=1}^{m}F_k(u)$$
and
$$\sum_{k=1}^{m}F_k(w)+[0, \sum_{k=1}^{m}\sqrt{\varepsilon_k}\|w-v\|_X]\preccurlyeq_{CW}\sum_{k=1}^{m}F_k(v).$$
Thus, again Remark \ref{re2.1}, we can see that
$$\sum_{k=1}^{m}F_k(w)+[0, \sum_{k=1}^{m}\sqrt{\varepsilon_k}(\|w-v\|_X+\|v-u\|_X)]\preccurlyeq_{CW}\sum_{k=1}^{m}F_k(u).$$
Since $\|w-u\|_X\leq\|w-v\|_X+\|v-u\|_X$, one has
$[0,\|w-u\|_X]\preccurlyeq_{CW}[0,\|w-v\|_X+\|v-u\|_X]$, which says that

$$\sum_{k=1}^{m}F_k(w)+[0, \sum_{k=1}^{m}\sqrt{\varepsilon_k}\|w-u\|_X]\preccurlyeq_{CW}\sum_{k=1}^{m}F_k(u).$$
Then $w\in T(u)$. Thus, for any $u\in T_0$, $v\in T(u)$, $T(v)\subseteq T(u)$. Next we choose the sequence $\{u_n\}\subseteq T_0$ satisfying $u_{n+1}\in T(u_n)$ for $n=1,2,\cdots$. Then
$$\sum_{k=1}^{m}F_k(u_{n+1})+[0, \sum_{k=1}^{m}\sqrt{\varepsilon_k}\|u_{n+1}-u_n\|_X]\preccurlyeq_{CW}\sum_{k=1}^{m}F_k(u_n),$$
which says that
$\sum_{k=1}^{m}F^c_k(u_{n+1})+\frac{1}{2}\sum_{k=1}^{m}\sqrt{\varepsilon_k}\|u_{n+1}-u_n\|_X\leq\sum_{k=1}^{m}F^c_k(u_n)$
and $\sum_{k=1}^{m}F^w_k(u_{n+1})+\frac{1}{2}\sum_{k=1}^{m}\sqrt{\varepsilon_k}\|u_{n+1}-u_n\|_X\leq\sum_{k=1}^{m}F^w_k(u_n)$.
Hence
$\sum_{k=1}^{m}F^c_k(u_{n+1})\leq\sum_{k=1}^{m}F^c_k(u_n)$ and $\sum_{k=1}^{m}F^w_k(u_{n+1})\leq\sum_{k=1}^{m}F^w_k(u_n)$. Thanks to $F^w_k(u_n)\geq0$, the sequence $\{\sum_{k=1}^{m}F^w_k(u_n)\}$ is convergent, which implies that $\|u_n-u_{n+1}\|_X\rightarrow0$. Therefore, employing Lemma \ref{le2.1}, there is $\overline{u}\in T_0$ satisfying $T(\overline{u})=\{\overline{u}\}$. Thus, one has
\begin{equation}\label{eq3.1}
\sum_{k=1}^{m}F_k(\overline{u})+[0, \sum_{k=1}^{m}\sqrt{\varepsilon_k}\|\overline{u}-x_0\|_X]\preccurlyeq_{CW}\sum_{k=1}^{m}F_k(x_0)
\end{equation}
and
\begin{equation}\label{eq3.2}
\sum_{k=1}^{m}F_k(u)+[0, \sum_{k=1}^{m}\sqrt{\varepsilon_k}\|u-\overline{u}\|_X]\not\preccurlyeq_{CW}\sum_{k=1}^{m}F_k(\overline{u}),\;\forall u\in T_0\backslash\{\overline{u}\}.
\end{equation}
It follows from \eqref{eq3.1} that
$$
\sum_{k=1}^{m}F^c_k(\overline{u})\leq\sum_{k=1}^{m}F^c_k(\overline{u})+\frac{1}{2}\sum_{k=1}^{m}\sqrt{\varepsilon_k}
\|\overline{u}-x_0\|_X\leq\sum_{k=1}^{m}F^c_k(x_0),
$$
$$
\sum_{k=1}^{m}F^w_k(\overline{u})\leq\sum_{k=1}^{m}F^w_k(\overline{u})+\frac{1}{2}\sum_{k=1}^{m}\sqrt{\varepsilon_k}
\|\overline{u}-x_0\|_X\leq\sum_{k=1}^{m}F^w_k(x_0).
$$
 If the part (a) is not true, then there exists $u'\in X$ satisfying
$\sum_{k=1}^{m}F_k(u')+[0,\sum_{k=1}^{m}\varepsilon_k]\prec_{CW}\sum_{k=1}^{m}F_k(\overline{u})$,
which says that
$\sum_{k=1}^{m}F^c_k(u')+\frac{1}{2}\sum_{k=1}^{m}\varepsilon_k<\sum_{k=1}^{m}F^c_k(\overline{u})$,
and $\sum_{k=1}^{m}F^w_k(u')+\frac{1}{2}\sum_{k=1}^{m}\varepsilon_k<\sum_{k=1}^{m}F^w_k(\overline{u})$.
Then $\sum_{k=1}^{m}F^c_k(u')+\frac{1}{2}\sum_{k=1}^{m}\varepsilon_k<\sum_{k=1}^{m}F^c_k(x_0)$
and  $\sum_{k=1}^{m}F^w_k(u')+\frac{1}{2}\sum_{k=1}^{m}\varepsilon_k<\sum_{k=1}^{m}F^w_k(x_0)$.
Now, due to the notion of $\prec_{CW}$, we know that
$\sum_{k=1}^{m}F_k(u')+[0,\sum_{k=1}^{m}\varepsilon_k]\prec_{CW}\sum_{k=1}^{m}F_k(x_0)$, which contradicts with the assumption. Thus, the part (a) is true.

Next, if the part (b) is not true, then $\|x_0-\overline{u}\|_X>\frac{\sum_{k=1}^{m}\varepsilon_k}{\sum_{k=1}^{m}\sqrt{\varepsilon_k}}$. Going back to (\ref{eq3.1}), we have
$$
\sum_{k=1}^{m}F^c_k(\overline{u})+\frac{1}{2}\sum_{k=1}^{m}\varepsilon_k
<\sum_{k=1}^{m}F^c_k(\overline{u})+\frac{1}{2}\sum_{k=1}^{m}\sqrt{\varepsilon_k}
\|\overline{u}-x_0\|_X\leq\sum_{k=1}^{m}F^c_k(x_0),
$$
$$
\sum_{k=1}^{m}F^w_k(\overline{u})+\frac{1}{2}\sum_{k=1}^{m}\varepsilon_k
<\sum_{k=1}^{m}F^w_k(\overline{u})+\frac{1}{2}\sum_{k=1}^{m}\sqrt{\varepsilon_k}
\|\overline{u}-x_0\|_X\leq\sum_{k=1}^{m}F^w_k(x_0),
$$
which means that
$$\sum_{k=1}^{m}F_k(\overline{u})+[0, \sum_{k=1}^{m}\varepsilon_k]\prec_{CW}\sum_{k=1}^{m}F_k(x_0),$$
which is a contradiction. Thus, the part (b) is true.

Finally, we show that the part (c) is true. To this end, we first prove that
\begin{equation}\label{eq3.3}
\sum_{k=1}^{m}F_k(u)+[0, \sum_{k=1}^{m}\sqrt{\varepsilon_k}\|u-\overline{u}\|_X]\not\preccurlyeq_{CW}\sum_{k=1}^{m}F_k(\overline{u}),\; \forall u\in X\backslash\{\overline{u}\}.
\end{equation}
Noting that (\ref{eq3.2}), if there is $v\in X\backslash T_0$ satisfying
$$
\sum_{k=1}^{m}F_k(v)+[0, \sum_{k=1}^{m}\sqrt{\varepsilon_k}\|v-\overline{u}\|_X]\preccurlyeq_{CW}\sum_{k=1}^{m}F_k(\overline{u}),
$$
then this together with (\ref{eq3.1}) gives
$$
\sum_{k=1}^{m}F_k(v)+[0, \sum_{k=1}^{m}\sqrt{\varepsilon_k}(\|v-\overline{u}\|_X+\|\overline{u}-x_0\|_X)]\preccurlyeq_{CW}\sum_{k=1}^{m}F_k(x_0).
$$
Thus, we deduce that
$$\sum_{k=1}^{m}F_k(v)+[0, \sum_{k=1}^{m}\sqrt{\varepsilon_k}\|v-x_0\|_X]\preccurlyeq_{CW}\sum_{k=1}^{m}F_k(x_0),$$
which contradicts with $v\not\in T_0$. Hence we show
(\ref{eq3.3}), which means that
$$
\sum_{k=1}^{m}F_k(u)+[0, \sum_{k=1}^{m}\sqrt{\varepsilon_k}\|u-\overline{u}\|_X]\not\prec_{CW}\sum_{k=1}^{m}F_k(\overline{u}),\; \forall u\in X\backslash\{\overline{u}\}.
$$
Moreover, we can deduce that $$
\sum_{k=1}^{m}F_k(u)+[0, \sum_{k=1}^{m}\sqrt{\varepsilon_k}\|u-\overline{u}\|_X]\not\prec_{CW}\sum_{k=1}^{m}F_k(\overline{u}),\; \forall u\in X,
$$
which shows that the part (c) is true.
\end{proof}
\begin{remark}\label{re3.1}
It should be noted that we do not need assume the lower boundedness of $F_k$ since the fact that $F^w_k(u)\geq0$.
\end{remark}

\begin{proposition}\label{p3.2}
Let $X$ be a real Banach space, and $F_k:X\to \mathcal{I}$ be $\preccurlyeq_{CW}$-lsc for all $k\in K$.
Then for $\epsilon>0$ and $x_0\in X$ satisfying $x_0\in M(F,X,\varepsilon)$, where
$\varepsilon:=(\epsilon,\epsilon,\cdots,\epsilon)\in\mathbb{R}^m_+$,
there is a point $\overline{u}\in X$ satisfying
\begin{itemize}
\item[(a)]  $\overline{u}\in M(F,X,\varepsilon)$, where $F(u):=(F_1(u),F_2(u),\cdots,F_m(u))$;
\item[(b)] $\|x_0-\overline{u}\|_X\leq\sqrt{\epsilon}$;
\item[(c)]  $\overline{u}\in QM(F,X,\varepsilon')$, where $\varepsilon':=(\sqrt{\epsilon},\sqrt{\epsilon},\cdots,\sqrt{\epsilon})\in\mathbb{R}^m_+$.
\end{itemize}
\end{proposition}
\begin{proof}
Now we first define the orders $\preccurlyeq$ and $\prec$ as follows: for any $u,v\in X$,
$$
F(u)\preccurlyeq F(v)\Leftrightarrow F_k(u)\preccurlyeq_{CW} F_k(v),\;\forall k\in K,\; \textrm{and}\;
F(u)\prec F(v)\Leftrightarrow F_k(u)\prec_{CW} F_k(v),\;\forall k\in K.
$$
Moreover, the addition between $F(u)$ and $F(v)$ is given by
$$F(u)+F(v):=(F_1(u)+F_1(v),F_2(u)+F_2(v),\cdots,F_m(u)+F_m(v)).$$
For $\epsilon>0$ and $x_0\in X$, let
$$\Gamma_0:=\{v\in X: F(v)+([0, \sqrt{\epsilon}\|v-x_0\|_X],\cdots,[0, \sqrt{\epsilon}\|v-x_0\|_X])\preccurlyeq F(x_0)\}.$$
Then a similar proofs of Proposition \ref{p3.1} applied to $\Gamma_0$, we can show the parts (a)-(c).
\end{proof}

When $K=\{1\}$, Propositions \ref{p3.1} and \ref{p3.2} are reduced to the following result.
\begin{corollary}\label{co3.1}
Let $X$ be a real Banach space, and $F:X\to \mathcal{I}$ be a $\preccurlyeq_{CW}$-lsc IVF.
Then for
$\epsilon>0$ and $x_0\in X$ satisfying $F(u)+[0,\epsilon]\not\prec_{CW}F(x_0)$ for all $u\in X$, there is a point $\overline{u}\in X$ satisfying
\begin{itemize}
\item[(a)]  $F(u)+[0,\epsilon]\not\prec_{CW}F(\overline{u}),\;\forall u\in X$;
\item[(b)] $\|x_0-\overline{u}\|_X\leq\sqrt{\epsilon}$;
\item[(c)] $F(u)+[0, \sqrt{\epsilon}\|u-\overline{u}\|_X]\not\prec_{CW}F(\overline{u})$, $\forall u\in X$.
\end{itemize}
\end{corollary}
\begin{remark}\label{re3.2}
We would like to mention that Corollary \ref{co3.1} is different from Theorem 3.1 of \cite{ZH} in the following aspects:
(i) the interval orders are different from the ones of \cite{ZH}; (ii) the lower boundedness of $F$ is removed.
\end{remark}

The following theorems give some sufficient conditions for ensuring the existence of weak $\varepsilon$-quasi-minimal solutions of MIOPs.
\begin{theorem}\label{th3.2}
Let $F_k:\mathbb{R}^n\to \mathcal{I}$ be $\preccurlyeq_{CW}$-lsc for all $k\in K$, and $g_j$ be lsc on $\mathbb{R}^n$ for all $j\in J$.
Then for every $\varepsilon:=(\varepsilon_1,\varepsilon_2,\cdots,\varepsilon_m)\in\mathbb{R}^m_+\backslash\{0\}$,
$QM(F,S,\varepsilon')\neq\emptyset$, where $\varepsilon':=(\sqrt{\varepsilon_1},\sqrt{\varepsilon_2},\cdots,\sqrt{\varepsilon_m})$.
\end{theorem}
\begin{proof}
Since $g_i$ is lsc on $\mathbb{R}^n$, the set $S$ is closed. Applying Proposition \ref{p3.1} to $S$,
there is $\overline{u}\in S$ such that
\begin{equation}\label{eq3.4}
\sum_{k=1}^{m}F_k(u)+[0, \sum_{k=1}^{m}\sqrt{\varepsilon_k}\|u-\overline{u}\|]\not\prec_{CW}\sum_{k=1}^{m}F_k(\overline{u}),\;\forall u\in S.
\end{equation}
Now we claim that $\overline{u}\in QM(F,S,\varepsilon')$. If not, then there exists $v\in S$
satisfying
$$F_k(v)+[0, \sqrt{\varepsilon_k}\|v-\overline{u}\|]\prec_{CW}F_k(\overline{u}),\;\forall k\in K.$$
By Remark \ref{re2.1}, we have
$$
\sum_{k=1}^{m}F_k(v)+[0, \sum_{k=1}^{m}\sqrt{\varepsilon_k}\|v-\overline{u}\|]\prec_{CW}\sum_{k=1}^{m}F_k(\overline{u}),
$$
which contradicts with (\ref{eq3.4}). Thus, $QM(F,S,\varepsilon')\neq\emptyset$.
\end{proof}

\begin{theorem}\label{th3.3}
Let $\varepsilon\in\mathbb{R}^m_+\backslash\{0\}$ and $\overline{u}\in S$. Assume that
$$
\sum_{k=1}^{m}
(F^c_k(u)+F^w_k(u))+\sum_{k=1}^{m}\varepsilon_k\|u-\overline{u}\|\geq \sum_{k=1}^{m}
(F^c_k(\overline{u})+F^w_k(\overline{u})),\;\forall u\in S.
$$
Then $QM(F,S,\varepsilon)\neq\emptyset$.
\end{theorem}
\begin{proof}
Assume that $\overline{u}\not\in QM(F,S,\varepsilon)$. Then there exists $v\in S$ satisfying
$$F_k(v)+[0,\varepsilon_k\|v-\overline{u}\|]\prec_{CW}F_k(\overline{u}),\;\forall k\in K.$$
Thanks to Remark \ref{re2.1}, one has
$$
\sum_{k=1}^{m}F_k(v)+[0, \sum_{k=1}^{m}\varepsilon_k\|v-\overline{u}\|]\prec_{CW}\sum_{k=1}^{m}F_k(\overline{u}).
$$
By the concept of $\prec_{CW}$, one has
$$
\sum_{k=1}^{m}(F^c_k(v)+F^w_k(v))+\sum_{k=1}^{m}\varepsilon_k\|v-\overline{u}\|<\sum_{k=1}^{m}(F^c_k(\overline{u})+F^w_k(\overline{u})),
$$
which is a contradiction. Thus, $\overline{u}\in QM(F,S,\varepsilon)$ and so $QM(F,S,\varepsilon)\neq\emptyset$.
\end{proof}

\begin{theorem}\label{th3.4}
Let $F_k:\mathbb{R}^n\to \mathcal{I}$ be $\preccurlyeq_{CW}$-lsc for all $k\in K$, and $g_j$ be lsc on $X$ for all $j\in J$.
Let $\varepsilon_0>0$ with $\varepsilon:=(\varepsilon_0,\varepsilon_0,\cdots,\varepsilon_0)\in\mathbb{R}^m_+$. Then there is $\overline{u}\in S$ satisfying
$ M(F,\overline{B}(\overline{u},\sqrt{\varepsilon_0}),\varepsilon)\neq\emptyset$.
\end{theorem}
\begin{proof}
For $\varepsilon=(\varepsilon_0,\varepsilon_0,\cdots,\varepsilon_0)\in\mathbb{R}^m_+$, by the proof of Theorem \ref{th3.2}, we know that there exists $\overline{u}\in S$ satisfying $\overline{u}\in QM(F,S,\varepsilon')$, where $\varepsilon':=(\sqrt{\varepsilon_0},\sqrt{\varepsilon_0},\cdots,\sqrt{\varepsilon_0})\in\mathbb{R}^m_+$.
Thanks to Proposition \ref{p2.1}, one has $\overline{u}\in M(F,\overline{B}(\overline{u},\sqrt{\varepsilon_0}),\varepsilon)$.
\end{proof}

\section{KKT optimality conditions}\label{sec4}
\setcounter{equation}{0}

In this section, we discuss KKT optimality conditions and approximate KKT optimality conditions for approximate solutions of MIOP.

Throughout this section, assume that $F_k$ and $g_j$ are locally Lipschitz continuous on $\mathbb{R}^n$ for all $k\in K$ and all $j\in J$. We first give the constraint qualification ensuring KKT optimality conditions for MIOP.
\begin{definition}\label{de4.1}(\cite{Dutta,Tuyen})
We say that the basic constraint qualification (for short, BCQ) holds at $\overline{u}\in S$ if there is no $\mu\in \mathbb{R}^p_+\backslash\{0\}$ with $\mu_j\geq0$ for all $j\in I(\overline{u})$ and $\mu_j=0$ for all $j\not\in I(\overline{u})$ such that
$$
0\in\sum_{j\in J}\mu_j\partial^0 g_j(\overline{u}),
$$
where $I(\overline{u}):=\{j\in J:g_j(\overline{u})=0\}$.
\end{definition}
\begin{remark}
(i) Clearly, if BCQ holds at $\overline{u}\in S$, then $I(\overline{u})$ is nonempty, i.e., $g_j(\overline{u})=0$ for some $j\in J$; (ii) We say that BCQ holds on a subset $S_0$ of $S$ if the condition BCQ holds at each point $u\in S_0$.
\end{remark}

To obtain some KKT necessary optimality conditions for approximate solutions of MIOP, we need the following lemma.
\begin{lemma}\label{le4.1}
Suppose that BCQ holds at $\overline{u}\in S$. Then $N_S(\overline{u})\subset\{\eta\xi:\eta\geq0,\xi\in co\{\partial^0g_j(\overline{u}):j\in I(\overline{u})\}\}.$
\end{lemma}
\begin{proof}
Define the function $g$ as $g(u):=\max\{g_j(u):j\in J\}$. Then by $S=\{u\in \mathbb{R}^n:g_j(u)\leq0,\forall j\in J\}$, one has $S=\{u\in \mathbb{R}^n:g(u)\leq0\}$. We notice that $g(\overline{u})=0$ due to
$g_j(\overline{u})=0$ for some $j\in J$.
Since BCQ holds at $\overline{u}\in S$, we have $0\not\in\partial^0g(\overline{u})$. Indeed, if $0\in\partial^0g(\overline{u})$,
then $0\in co\{\partial^0g_j(\overline{u}):j\in I(\overline{u})\}$. This implies that there exists $j\in I(\overline{u})$, $\mu_j\geq0$ not all zero,  satisfying
$0\in\sum_{j\in I(\overline{u})}\mu_j\partial^0 g_j(\overline{u})$.
Now we set $\mu_j=0$ for $j\not\in I(\overline{u})$ and so $0\in\sum_{j\in J}\mu_j\partial^0 g_j(\overline{u})$,
which contradicts with the condition BCQ.
Thus, $0\not\in\partial^0g(\overline{u})$ and it follows from Theorem 10.42 in \cite{Clarke2} that
$N_S(\overline{u})\subset\{\eta\xi:\eta\geq0,\xi\in\partial^0g(\overline{u})\},$
which says that $N_S(\overline{u})\subset\{\eta\xi:\eta\geq0,\xi\in co\{\partial^0g_j(\overline{u}):j\in I(\overline{u})\}\}$.
\end{proof}

Now we first give the KKT necessary optimality condition of weak $\varepsilon$-minimal solutions for MIOP.
\begin{theorem}\label{th4.1}
Let $\varepsilon\in\mathbb{R}^m_+\backslash\{0\}$ and $\overline{u}\in M(F,S,\varepsilon)$. Given $\delta>0$, assume that BCQ holds on $\overline{B}(\overline{u},\delta)$. Then there exist $x_\delta\in \overline{B}(\overline{u},\delta)$ and $\lambda\in\mathbb{R}^m_+$, $\mu\in\mathbb{R}^p_+$ with $\sum_{k\in K}\lambda_k=1$, satisfying
\begin{equation}\label{eq4.1}
\begin{split}
0\in&\sum_{k\in K}\lambda_k\partial^0 F_k(x_\delta)+\sum_{j\in I(x_\delta)}\mu_j\partial^0 g_j(x_\delta)+\frac{1}{\delta}\max_{k\in K}\{\varepsilon_k\}\mathbb{B}^n,\\
&\mu_jg_j(x_\delta)=0, \;\forall j\in J.
\end{split}
\end{equation}
\end{theorem}
\begin{proof}
Since $\overline{u}\in M(F,S,\varepsilon)$, there is no $u\in S$ such that $F_k(u)+[0,\varepsilon_k]\prec_{CW} F_k(\overline{u})$ for all $k\in K$. This says that
for every $u\in S$, there is $k\in K$ such that $F^c_k(u)+\frac{1}{2}\varepsilon_k-F^c_k(\overline{u})\geq0$ or $F^w_k(u)+\frac{1}{2}\varepsilon_k-F^w_k(\overline{u})\geq0$.
For any $u\in \mathbb{R}^n$, we define the function $\phi:\mathbb{R}^n\to\mathbb{R}$ by
$$\phi(u):=\max\{F^c_k(u)+\frac{1}{2}\varepsilon_k-F^c_k(\overline{u}),F^w_k(u)+\frac{1}{2}\varepsilon_k-F^w_k(\overline{u}):k\in K\}.$$
Then $\phi$ is locally Lipschitz continuous and lower bounded on $S$ and $\phi(u)\geq0$ for all $u\in S$. Thanks to $\phi(\overline{u})=\frac{1}{2}\max\{\varepsilon_k:k\in K\}$, one has
$$\phi(\overline{u})\leq\inf_{u\in S}\phi(u)+\frac{1}{2}\max_{k\in K}\{\varepsilon_k\}.$$
Then, applying EVP (\cite{Ekeland1}) on $S$, for every $\delta>0$, we can find $x_\delta\in S$ such that
 $\|x_\delta-\overline{u}\|\leq\frac{1}{2}\delta$ and
 $\phi(x_\delta)\leq\phi(u)+\frac{1}{\delta}\max_{k\in K}\{\varepsilon_k\}\|u-x_\delta\|$ for all $u\in S$.
Thus, we can define the real-valued function $\Phi$ as follows
$$\Phi(u):=\phi(u)+\frac{1}{\delta}\max_{k\in K}\{\varepsilon_k\}\|u-x_\delta\|,\;\forall u\in\mathbb{R}^n.$$
Clearly, $x_\delta$ is a minimizer of $\Phi(\cdot)+(L+1)d_S(\cdot)$, where $L$ is a Lipschitz constant for $\Phi$.
Thus, $0\in\partial^0\Phi(x_\delta)+(L+1)\partial^0d_S(x_\delta)$. By Proposition 11.2.4 in \cite{Schi}, we know that $N_S(x_\delta)=cl(\mathbb{R}_+\partial^0d_S(x_\delta))$. Hence the sum rule gives
$0\in \partial^0\phi(x_\delta)+N_S(x_\delta)+\frac{1}{\delta}\max_{k\in K}\{\varepsilon_k\}\mathbb{B}^n$.
Since
$\partial^0\phi(x_\delta)\subset \textrm{co}\{\partial^0 F^c_r(x_\delta),\partial^0 F^w_s(x_\delta):
F^c_r(x_\delta)+\frac{1}{2}\varepsilon_r-F^w_r(\overline{u})=\phi(x_\delta),
F^w_s(x_\delta)+\frac{1}{2}\varepsilon_s-F^w_s(\overline{u})=\phi(x_\delta), r,s\in K\}$, it follows from the notion of $\partial^0 F_k(x_\delta)$ that $\partial^0\phi(x_\delta)\subset \textrm{co}\{\partial^0 F_k(x_\delta): k\in K\}$.
Then by Lemma \ref{le4.1} that there exist $\lambda\in\mathbb{R}^m_+$ with $\sum_{k\in K}\lambda_k=1$, $\mu_j\geq0$ not all zero for $j\in I(x_\delta)$ and $\mu_j=0$ for $j\not\in I(x_\delta)$, satisfying
$$0\in\sum_{k\in K}\lambda_k\partial^0 F_k(x_\delta)+\sum_{j\in I(x_\delta)}\mu_j\partial^0 g_j(x_\delta)+\frac{1}{\delta}\max_{k\in K}\{\varepsilon_k\}\mathbb{B}^n,$$
$$\mu_jg_j(x_\delta)=0, \;\forall j\in J.$$
\end{proof}

We now derive the KKT  necessary optimality conditions of weak minimal solutions for MIOP.
\begin{theorem}\label{th4.2}
Let $\overline{u}\in M(F,S)$. Assume that BCQ holds at $\overline{u}\in S$. Then there exists $\lambda\in\mathbb{R}^m_+$, $\mu\in\mathbb{R}^p_+$, $\sum_{k\in K}\lambda_k=1$, satisfying
\begin{equation}\label{eq4.2}
\begin{split}
0\in&\sum_{k\in K}\lambda_k\partial^0 F_k(\overline{u})+\sum_{j\in I(\overline{u})}\mu_j\partial^0 g_j(\overline{u}),\\
&\mu_jg_j(\overline{u})=0, \;\forall j\in J.
\end{split}
\end{equation}
\end{theorem}
\begin{proof}
Since $\overline{u}\in M(F,S)$, there is no $u\in S$ satisfying $F_k(u)\prec_{CW} F_k(\overline{u})$ for all $k\in K$. Thus,
for each $u\in S$, there is $k\in K$ such that $F^c_k(u)-F^c_k(\overline{u})\geq0$ or $F^w_k(u)-F^w_k(\overline{u})\geq0$.
For any $u\in \mathbb{R}^n$, we define the function $\phi:\mathbb{R}^n\to\mathbb{R}$ by
$$\phi(u):=\max\{F^c_k(u)-F^c_k(\overline{u}),F^w_k(u)-F^w_k(\overline{u}):k\in K\}.$$
Note that $\phi$ is locally Lipschitz continuous on $\mathbb{R}^n$ and $\phi(u)\geq0$ for all $u\in S$. By $\phi(\overline{u})=0$, one has $\phi(u)\geq\phi(\overline{u})$ for all $u\in S$.
Hence $\overline{u}$ is a minimizer of $\phi(\cdot)+(L+1)d_S(\cdot)$, where $L$ is a Lipschitz constant for $\phi$.
Hence $0\in\partial^0\phi(\overline{u})+(L+1)\partial^0d_S(\overline{u})$. Employing Proposition 11.2.4 in \cite{Schi}, we know that $N_S(\overline{u})=cl(\mathbb{R}_+\partial^0d_S(\overline{u}))$. This implies that
$0\in co\{\partial^0F^c_r(\overline{u}),\partial^0 F^w_s(\overline{u}):F^c_r(u)=F^c_r(\overline{u}),F^w_s(u)=F^w_s(\overline{u}),r,s\in K\}+N_S(\overline{u})$.
Then it follows from the notion of $\partial^0 F_k(\overline{u})$ and Lemma \ref{le4.1} that there exist $\lambda\in\mathbb{R}^m_+$ with $\sum_{k\in K}\lambda_k=1$, $\mu_j\geq0$ not all zero for $j\in I(\overline{u})$ and $\mu_j=0$ for $j\not\in I(\overline{u})$, satisfying
$$0\in\sum_{k\in K}\lambda_k\partial^0 F_k(\overline{u})+\sum_{j\in I(\overline{u})}\mu_j\partial^0 g_j(\overline{u}),$$
$$\mu_jg_j(\overline{u})=0, \;\forall j\in J.$$
\end{proof}
\begin{remark}\label{re4.2}
\begin{itemize}
\item[(i)] A point $\overline{u}\in S$ is called a KKT point if $\overline{u}$ satisfies (\ref{eq4.2});

\item[(ii)] If $\overline{u}$ is a weak minimizer of $F$ on $\mathbb{R}^n$, then $0\in\sum_{k\in K}\lambda_k\partial^0 F_k(\overline{u})$
for some $\lambda\in\mathbb{R}^m_+$ and $\sum_{k\in K}\lambda_k=1$.
\end{itemize}
\end{remark}

We present an example to illustrate Theorem \ref{th4.2}.
\begin{example}
Consider the following multiobjective interval-valued optimization problem:
$$
\textrm{Min}\; F(u):=(F_1(u),F_2(u))\quad \textrm{s.t.} \;\; u\in S,
$$
where $S:=\{u\in \mathbb{R}:g_1(u)=-u\leq0, g_2(u)=-u-1\leq0\}$, $F_1(u):=[|u|,|u|+1]$ and $F_2(u):=[2|u|,2|u|+2]$ for all $u\in\mathbb{R}$,
Let $\overline{u}=0$. Clearly, $\overline{u}=0$ is a weak minimal solution for MIOP with $I(\overline{u})=\{j:g_j(\overline{u})=0,j\in\{1,2\}\}=\{1\}$. We also can check that $\partial^0 F_1(0)=[-1,1]$,
$\partial^0 F_2(0)=[-2,2]$, and $\partial^0 g_j(0)=\{-1\}$ for $j=1,2$ . Moreover, we can see that BCQ holds at $\overline{u}$.
Now setting $\lambda_1=\lambda_2=\frac{1}{2}$ and $\mu_1=\frac{1}{2}$, $\mu_2=0$, we can easily check that $\overline{u}$ satisfies (\ref{eq4.2}).

\end{example}

Next we give the KKT necessary optimality condition of weak $\varepsilon$-quasi-minimal solutions for MIOP.
\begin{corollary}\label{co4.1}
Let $\varepsilon\in\mathbb{R}^m_+\backslash\{0\}$ and  $\overline{u}\in QM(F,S,\varepsilon)$. Assume that BCQ holds at $\overline{u}\in S$. Then there exists $\lambda\in\mathbb{R}^m_+$, $\mu\in\mathbb{R}^p_+$, $\sum_{k\in K}\lambda_k=1$, such that
\begin{equation}\label{eq4.3}
\begin{split}
0\in&\sum_{k\in K}\lambda_k\partial^0 F_k(\overline{u})+\sum_{j\in I(\overline{u})}\mu_j\partial^0 g_j(\overline{u})+\sum_{k\in K}\lambda_k\varepsilon_k \mathbb{B}^n,\\
&\mu_jg_j(\overline{u})=0, \;\forall j\in J.
\end{split}
\end{equation}
\end{corollary}
\begin{proof}
Since $\overline{u}\in QM(F,S,\varepsilon)$, we can see that $\overline{u}$ is a weak minimal solution of the following problem:
$$
\quad \textrm{Min}\; \overline{F}(u):=(F_1(u)+[0,\varepsilon_1\|u-\overline{u}\|],\cdots,
F_m(u)+[0,\varepsilon_m\|u-\overline{u}\|])\quad \textrm{s.t.} \;\; u\in S,
$$
By Remark \ref{re2.3}, one has $\partial^0(F_k(\cdot)+[0,\varepsilon_k\|\cdot-\overline{u}\|])\subset\partial^0F_k(\cdot)+
\varepsilon_k\mathbb{B}^n$. Thus, by Theorem \ref{th4.2}, there is $\lambda\in\mathbb{R}^m_+$, $\mu\in\mathbb{R}^p_+$, $\sum_{k\in K}\lambda_k=1$, such that
$$0\in\sum_{k\in K}\lambda_k\partial^0 F_k(\overline{u})+\sum_{j\in I(\overline{u})}\mu_j\partial^0 g_j(\overline{u})+\sum_{k\in K}\lambda_k\varepsilon_k \mathbb{B}^n,$$
$$\mu_jg_j(\overline{u})=0, \;\forall j\in J.$$
\end{proof}

The following concept is inspired from \cite{CK,Son}.
\begin{definition}\label{de4.2}
Let $X_0$ be a nonempty of $\mathbb{R}^n$. $(F,g_J)$ is called generalized convex at $u_0\in X_0$ if for any $u\in X_0$, $u^*_k\in\partial^0F_k(u_0)$, $k\in K$ and $z^*_j\in \partial^0 g_j(u_0)$, $j\in J$, there exists $v\in \mathbb{R}^n$ such that
\begin{equation*}
\begin{split}
&\langle u^*_k,v\rangle\leq F^c_k(u)-F^c_k(u_0)+F_k^w(u)-F_k^w(u_0),\; \forall k\in K,\\
&\langle z^*_j,v\rangle\leq g_j(u)-g_j(u_0),\forall j\in J,\\
&\langle y^*,v\rangle\leq\|u-u_0\|,\forall y^*\in\mathbb{B}^n.
\end{split}
\end{equation*}
\end{definition}

We will give an example to illustrate Definition \ref{de4.2}.
\begin{example}
 Consider the following multiobjective interval-valued function defined as:
$$
 F(u):=(F_1(u),F_2(u)),
$$
where $F_1(u):=[|u|,|u|+1]$, $F_2(u):=[2|u|,2|u|+1]$, and $g(u):=-u$ for all $u\in \mathbb{R}$.
Let $X_0=\mathbb{R}$ and $u_0=0$. Clearly, $F^c_1(u)=|u|+\frac{1}{2}$, $F^w_1(u)=1$, $F^c_2(u)=2|u|+\frac{1}{2}$, $F^w_2(u)=1$, $\partial^0 F_1(u_0)=[-1,1]$,
$\partial^0 F_2(u_0)=[-2,2]$, and $\partial^0 g(u_0)=\{-1\}$.
For any $u\in X_0$, $u^*_k\in\partial^0F_k(u_0)$, $k\in \{1,2\}$ and $z^*\in \partial^0 g(u_0)$, choosing $v:=g(u_0)-g(u)$, we can check that
$$ u^*_k\cdot v\leq F^c_k(u)-F^c_k(u_0)+F^w_k(u)-F^w_k(u_0),\;  z^*\cdot v\leq g(u)-g(u_0).$$
We also have $y^*\cdot v\leq|u-u_0|$ for all $y^*\in [-1,1]$. Thus, $(F,g)$ is generalized convex at $u_0\in X_0$.

\end{example}

Now we give the KKT sufficient optimality condition of weak $\varepsilon$-quasi-minimal solutions for MIOP involving generalized convexity condition.
\begin{theorem}\label{th4.3}
Let $\varepsilon\in\mathbb{R}^m_+\backslash\{0\}$ and $\overline{u}\in S$ satisfy (\ref{eq4.3}).
Assume that $(F,g_J)$ is generalized convex at $\overline{u}$. Then $\overline{u}\in QM(F,S,\varepsilon)$.
\end{theorem}
\begin{proof}
Since $\overline{u}\in S$ satisfies (\ref{eq4.3}), there exist $\lambda_k\geq0$, $x^*_k\in\partial^0 F_k(\overline{u})$, $k\in K$ with $\sum_{k\in K}\lambda_k=1$, and $\mu_j\geq0$, $z^*_j\in \partial^0 g_j(\overline{u})$, $j\in J$ with $\mu_jg_j(\overline{u})=0$, and $y^*\in \mathbb{B}^n$, satisfying
\begin{equation}\label{eq4.4}
0=\sum_{k\in K}\lambda_k x^*_k+\sum_{j\in I(\overline{u})}\mu_jz^*_j+\sum_{k\in K}\lambda_k\varepsilon_k y^*.
\end{equation}
Now suppose that $\overline{u}\not\in QM(F,S,\varepsilon)$. Then there is $u'\in S$ such that
$$
F_k(u')+[0,\varepsilon_k\|u'-\overline{u}\|]\prec_{CW}F_k(\overline{u}),\;\forall k\in K.
$$
By Remark \ref{re2.1}, we get
$\sum_{k\in K}\lambda_kF_k(u')+[0,\sum_{k\in K}\lambda_k\varepsilon_k\|u'-\overline{u}\|]\prec_{CW}\sum_{k\in K}\lambda_kF_k(\overline{u})$.
which means that
$$\sum_{k\in K}\lambda_kF^c_k(u')+\frac{1}{2}\sum_{k\in K}\lambda_k\varepsilon_k\|u'-\overline{u}\|<\sum_{k\in K}\lambda_kF^c_k(\overline{u}),$$
$$
\sum_{k\in K}\lambda_kF^w_k(u')+\frac{1}{2}\sum_{k\in K}\lambda_k\varepsilon_k\|u'-\overline{u}\|<\sum_{k\in K}\lambda_kF^w_k(\overline{u}).
$$
Then
\begin{equation}\label{eq4.5}
\sum_{k\in K}\lambda_k(F^c_k(u')+F^w_k(u'))+\sum_{k\in K}\lambda_k\varepsilon_k\|u'-\overline{u}\|<\sum_{k\in K}\lambda_k(F^c_k(\overline{u})
+F^w_k(\overline{u})).
\end{equation}
Since $(F,g_J)$ is generalized convex at $\overline{u}$, by (\ref{eq4.4}), there exists $v\in \mathbb{R}^n$ such that
\begin{equation*}
\begin{split}
0&=\sum_{k\in K}\lambda_k\langle x^*_k,v\rangle+\sum_{j\in I(\overline{u})}\mu_j\langle z^*_j,v\rangle+\sum_{k\in K}\lambda_k\varepsilon_k \langle y^*,v\rangle\\
&\leq\sum_{k\in K}\lambda_k( F^c_k(u')-F^c_k(\overline{u})+ F^w_k(u')-F^w_k(\overline{u}))+\sum_{k\in K}\lambda_k\varepsilon_k\|u'-\overline{u}\|.
\end{split}
\end{equation*}
Hence we can see that
$$
\sum_{k\in K}\lambda_k(F^c_k(\overline{u})
+F^w_k(\overline{u}))\leq\sum_{k\in K}\lambda_k(F^c_k(u')+F^w_k(u'))+\sum_{k\in K}\lambda_k\varepsilon_k\|u'-\overline{u}\|,
$$
which contradicts with (\ref{eq4.5}).
\end{proof}

By the notion of modified $\epsilon$-KKT point of MOPs in \cite{KSDD}, we also introduce a modified $\epsilon$-KKT point of MIOP.
\begin{definition}\label{de4.3}
For any given $\epsilon\geq0$, we say that $x_0\in S$ is a modified $\epsilon$-KKT point of MIOP if there is $x_\epsilon$ such that
$\|x_\epsilon-x_0\|\leq\sqrt{\epsilon}$ and there exist $u_k\in\partial^0F_k(x_\epsilon)$ for all $k\in K$, $v_j\in\partial^0g_j(x_\epsilon)$ for all $j\in J$, $\lambda\in\mathbb{R}^m_+$ with $\sum_{k\in K}\lambda_k=1$ and $\mu\in\mathbb{R}^p_+$ such that
$$
\|\sum_{k\in K}\lambda_k u_k+\sum_{j\in J}\mu_jv_j\|\leq\sqrt{\epsilon},
$$
$$
\sum_{j\in J}\mu_jg_j(x_0)\geq-\epsilon.
$$
\end{definition}
\begin{remark}\label{re4.3}
Clearly, (i) if $\epsilon=0$, then the modified $\epsilon$-KKT point of MIOP reduces to the KKT point of MIOP;
(ii) if $F_k$ is a real-valued function for all $k\in K$, then Definition \ref{de4.3} coincides with the notion of a modified $\epsilon$-KKT point in \cite{KSDD}.
\end{remark}

\begin{lemma}\label{le4.2}(\cite{Clarke2,Clarke3,Clarke4})
Let $S_1$ and $S_2$ be two closed subsets of $\mathbb{R}^n$ and $u\in S_1\bigcap S_2$. Assume that $-N_{S_1}(u)\bigcap N_{S_2}(u)=\{0\}$. Then $N_{S_1\cap S_2}(u)\subset N_{S_1}(u)+ N_{S_2}(u)$.
\end{lemma}

The following result gives the approximate KKT sufficient optimality condition for locally weak minimal solutions of MIOP.
\begin{theorem}\label{th4.4}
Let $\overline{u}\in S$ be a locally weak minimal solution of MIOP and BCQ hold on a neighbourhood of $\overline{u}$. Suppose that $\epsilon_i>0$ for $i=1,2,\cdots$ with $\epsilon_i\rightarrow0$.
Then for any sequence $\{x_i\}$ in $S$ with $x_i\rightarrow \overline{u}$ as $i\rightarrow+\infty$, there is a subsequence $\{z_i\}$ of $\{x_i\}$ such that for every $z_i$, there is $y_i\in S$ satisfying
\begin{itemize}
\item[(i)] $\|z_i-y_i\|\leq\sqrt{\epsilon_i}$;

\item[(ii)] there exist $u_k^i\in\partial^0F_k(y_i)$ for all $k\in K$ and $v_j^i\in\partial^0g_j(y_i)$ for all $j\in J$ satisfying
\begin{equation}\label{eq4.6}
\|\sum_{k\in K}\lambda_k^i u_k^i+\sum_{j\in J}\mu_j^iv_j^i\|\leq\sqrt{\epsilon_i},
\end{equation}
\begin{equation}\label{eq4.7}
\sum_{j\in J}\mu_j^ig_j(y_i)=0,
\end{equation}
where $\lambda^i:=(\lambda_1^i,\lambda_2^i,\cdots,\lambda_m^i)\in\mathbb{R}_+^m$ with $\sum_{k\in K}\lambda_k^i=1$ and
$\mu^i:=(\mu_1^i,\mu_2^i,\cdots,\mu_p^i)\in\mathbb{R}_+^p$.
\end{itemize}
\end{theorem}
\begin{proof}
Because $\overline{u}\in S$ is a locally weak minimal solution of MIOP, we can find a closed neighbourhood $U_{\delta}(\overline{u}):=\{z\in \mathbb{R}^n:\|z-\overline{u}\|\leq \delta\}$ of $\overline{u}$ for some $\delta>0$ and there is no $z\in S\bigcap U_{\delta}(\overline{u})$ satisfying
\begin{equation}\label{eq4.8}
F_k(z)\prec_{CW} F_k(\overline{u}),\; \forall k\in K.
\end{equation}
Let $X:=S\bigcap U_{\delta}(\overline{u})$. Thus, $X$ is closed and bounded. For any sequence $\{x_i\}$ in $S$ with $x_i\rightarrow \overline{u}$ as $i\rightarrow+\infty$, we have $x_i\in U_\delta(\overline{u})$ for $i$ large enough, which means that $x_i\in X$ for $i$ large enough.

Now we first want to show that there exists a subsequence $\{z_i\}$ of $\{x_i\}$ such that $z_i\in X$ is a weak $\overline{\varepsilon}_i$-minimal solution for MIOP
under the feasible set $X$, where $\overline{\varepsilon}_i:=(\epsilon_i,\epsilon_i,\cdots,\epsilon_i)\in\mathbb{R}^m_+$.
In fact, since $F_k$ is locally Lipschitz continuous on $\mathbb{R}^n$ for $k\in K$, by Remark \ref{re2.2}, we have
$$F^c_k(x_i)\rightarrow F^c_k(\overline{u}),\;F^w_k(x_i)\rightarrow F^w_k(\overline{u}) \; \textrm{as}\; i\rightarrow+\infty.$$
Hence, for given
$\epsilon_1>0$, for every $k$, there is $N_1^k>0$ such that, for all $i\geq N_1^k$,
$$
|F^c_k(x_i)-F^c_k(\overline{u})|<\frac{1}{2}\epsilon_1,\; |F^w_k(x_i)-F^w_k(\overline{u})|<\frac{1}{2}\epsilon_1.
$$
Thus, we can choose $N_1:=\max\{N_1^1,N_1^2,\cdots,N_1^m\}$ such that, for all $k\in K$, $i\geq N_1$,
$$
|F^c_k(x_i)-F^c_k(\overline{u})|<\frac{1}{2}\epsilon_1,\; |F^w_k(x_i)-F^w_k(\overline{u})|<\frac{1}{2}\epsilon_1.
$$
Now we choose $z_1=x_{N_1}$ and thus $|F^c_k(z_1)-F^c_k(\overline{u})|<\frac{1}{2}\epsilon_1$ with $|F^w_k(z_1)-F^w_k(\overline{u})|<\frac{1}{2}\epsilon_1$. This means that
\begin{equation}\label{eq4.9}
F^c_k(z_1)<F^c_k(\overline{u})+\frac{1}{2}\epsilon_1,\;F^w_k(z_1)<F^w_k(\overline{u})+\frac{1}{2}\epsilon_1,\;\forall k\in K.
\end{equation}
Thanks to (\ref{eq4.9}), we can deduce that there is no $z\in X$ satisfying $F_k(z)+[0,\epsilon_1]\prec_{CW} F_k(z_1)$ for all $k\in K$.
If not, then there exists $u'\in X$ satisfying $F_k(u')+[0,\epsilon_1]\prec_{CW} F_k(z_1)$ for all $k\in K$, which means that
$F^c_k(u')+\frac{1}{2}\epsilon_1<F^c_k(z_1)$ and $F^w_k(u')+\frac{1}{2}\epsilon_1<F^w_k(z_1)$.
It follows from (\ref{eq4.9}) that $F^c_k(u')<F^c_k(\overline{u})$ and $F^w_k(u')<F^w_k(\overline{u})$ for all $k\in K$.
So we have  $F_k(u')\prec_{CW}F_k(\overline{u})$ for all $k\in K$, which contradicts with (\ref{eq4.8}).

Letting $\epsilon_2<\epsilon_1$, a similar argument can be applied to $\{x_{N_1},x_{{N_1}+1},\cdots\}$. Then we can choose $N_2$ with
$N_2>N_1$ and set $z_2=x_{N_2}$ such that there is no $z\in X$ satisfying $F_k(z)+[0,\epsilon_2]\prec_{CW} F_k(z_2)$ for all $k\in K$.
Proceeding in this way, we can construct the subsequence $\{z_i\}$ of $\{x_i\}$ such that $z_i\in X$ and there is no $z\in X$ satisfying
$F_k(z)+[0,\epsilon_i]\prec_{CW} F_k(z_i)$ for all $k\in K$. Thus, we obtain that $z_i$ is a weak $\overline{\varepsilon}_i$-minimal solution for (MIOP) under the feasible set $X$.

For any $k\in K$, since $F_k$ is locally Lipschitz continuous on $\mathbb{R}^n$, by Remark \ref{re2.2}, $F_k$ is $\preccurlyeq_{CW}$-lsc on $X$, Now applying the multiobjective interval-valued version of EVP (Proposition \ref{p3.2}) on $X$, we can see that for every $z_i\in X$, there exists $y_i\in X$ satisfying $\|y_i-z_i\|\leq\sqrt{\epsilon_i}$, and there is no $z\in X$ satisfying
\begin{itemize}
\item[(a)] $F_k(z)+[0,\epsilon_i]\prec_{CW} F_k(y_i),\; \forall k\in K$;

\item[(b)] $F_k(z)+[0,\sqrt{\epsilon_i}\|z-y_i\|]\prec_{CW} F_k(y_i), \; \forall k\in K$.
\end{itemize}
As a result, we know that $y_i$ is a weak minimal solution for the problem as follows:
$$
\quad \textrm{Min}\; \overline{F}(u):=(F_1(u)+[0,\sqrt{\epsilon_i}\|u-y_i\|],\cdots,
F_m(u)+[0,\sqrt{\epsilon_i}\|u-y_i\|])\quad \textrm{s.t.} \;\; u\in X.
$$
By the proofs of Theorem \ref{th4.2} and Corollary \ref{co4.1}, there exists $\lambda^i\in\mathbb{R}^m_+$ with $\sum_{k\in K}\lambda_k^i=1$, such that
\begin{equation}\label{eq4.10}
0\in\sum_{k\in K}\lambda_k^i\partial^0 F_k(y_i)+\sqrt{\epsilon_i} \mathbb{B}^n+N_X(y_i).
\end{equation}
Noting that $x_i\rightarrow \overline{u}$ and the subsequence $\{z_i\}$ of $\{x_i\}$, one has $z_i\in S\bigcap\textrm{int}U_{\delta}(\overline{u})$ for $i$ large enough, where $\textrm{int}U_{\delta}(\overline{u})=\{z\in \mathbb{R}^n:\|z-\overline{u}\|<\delta\}$. Since $\|y_i-z_i\|\leq\sqrt{\epsilon_i}$, we
have
$$\|y_i-\overline{u}\|\leq\|y_i-z_i\|+\|z_i-\overline{u}\|\leq\sqrt{\epsilon_i}+\|z_i-\overline{u}\|.$$
When $i\rightarrow+\infty$, we can see $y_i\in\textrm{int}U_{\delta}(\overline{u})$ for $i$ large enough. Without loss of generality, we can assume that $y_i\in\textrm{int}U_{\delta}(\overline{u})$.
According to the convexity of $U_{\delta}(\overline{u})$, we can easily see $N_{U_{\delta}(\overline{u})}(y_i)=\{0\}$ and $-N_{S}(y_i)\bigcap N_{U_{\delta}(\overline{u})}(y_i)=\{0\}$. Thanks to Lemma \ref{le4.2}, we obtain
$N_X(y_i)=N_{S\cap U_{\delta}(\overline{u})}(y_i)\subset N_S(y_i)$.
Thus, the relation (\ref{eq4.10}) implies that
\begin{equation}\label{eq4.11}
0\in\sum_{k\in K}\lambda_k^i\partial^0 F_k(y_i)+\sqrt{\epsilon_i} \mathbb{B}^n+N_S(y_i).
\end{equation}
Since the condition BCQ holds on a neighbourhood of $\overline{u}$, it follows from Lemma \ref{le4.1} that
$$N_S(y_i)\subset\{\eta\xi:\eta\geq0,\xi \in co\{\partial^0g_j(y_i):j\in I(y_i)\}.$$
Combining the above inclusion relation with (\ref{eq4.11}), there exist $\mu_j^i\geq0$ not all zero for $j\in I(y_i)$ and $\mu_j^i=0$ for $j\not\in I(y_i)$, $\lambda^i\in\mathbb{R}^m_+$, $\sum_{k\in K}\lambda_k^i=1$, such that (\ref{eq4.6}) and (\ref{eq4.7}) hold.
\end{proof}
\begin{remark}
Theorem \ref{th4.4} shows that for every a locally minimal solution of MIOP, there exists a sequence converging to the point and the sequence has a subsequence of approximate solutions which satisfies certain approximate KKT conditions.
\end{remark}
\begin{remark}
It should be noted that Theorem \ref{th4.4} is different from Theorem 3.6 of \cite{KSDD} in the following aspects: (i) Theorem \ref{th4.4} does not require any convexity; (ii) Theorem \ref{th4.4} is established in multi-interval cases.
\end{remark}

We present the approximate KKT condition for locally weak $\varepsilon$-quasi-minimal solutions of MIOP.
\begin{corollary}\label{co4.2}
Let $\varepsilon\in\mathbb{R}^m_+\backslash\{0\}$ and let $\overline{u}\in S$ be a locally weak $\varepsilon$-quasi-minimal solution of MIOP and BCQ hold on a neighbourhood of $\overline{u}$. Suppose that $\epsilon_i>0$ for $i=1,2,\cdots$ with $\epsilon_i\rightarrow0$.
Then for any sequence $\{x_i\}$ in $S$ with $x_i\rightarrow \overline{u}$ as $i\rightarrow+\infty$, there is a subsequence $\{z_i\}$ of $\{x_i\}$ such that for every $z_i$, there is $y_i\in S$ satisfying
\begin{itemize}
\item[(i)] $\|z_i-y_i\|\leq\sqrt{\epsilon_i}$;

\item[(ii)] there exist $u_k^i\in\partial^0F_k(y_i)+
\varepsilon_k\mathbb{B}^n$ for all $k\in K$ and $v_j^i\in\partial^0g_j(y_i)$ for all $j\in J$ satisfying
\begin{equation}\label{eq4.12}
\|\sum_{k\in K}\lambda_k^i u_k^i+\sum_{j\in J}\mu_j^iv_j^i\|\leq\sqrt{\epsilon_i},
\end{equation}
\begin{equation}\label{eq4.13}
\sum_{j\in J}\mu_j^ig_j(y_i)=0,
\end{equation}
where $\lambda^i:=(\lambda_1^i,\lambda_2^i,\cdots,\lambda_m^i)\in\mathbb{R}_+^m$ with $\sum_{k\in K}\lambda_k^i=1$ and
$\mu^i:=(\mu_1^i,\mu_2^i,\cdots,\mu_p^i)\in\mathbb{R}_+^p$.
\end{itemize}
\end{corollary}
\begin{proof}
Since $\overline{u}\in S$ is a locally weak $\varepsilon$-quasi-minimal solution of MIOP, we can see that $\overline{u}$ is a locally weak minimal solution of the following problem:
$$
\quad \textrm{Min}\; \overline{F}(u):=(F_1(u)+[0,\varepsilon_1\|u-\overline{u}\|],\cdots,
F_m(u)+[0,\varepsilon_m\|u-\overline{u}\|])\quad \textrm{s.t.} \;\; u\in S.
$$
Due to Remark \ref{re2.3}, one has $$\partial^0(F_k(u)+[0,\varepsilon_k\|u-\overline{u}\|])\subset\partial^0F_k(u)+
\varepsilon_k\mathbb{B}^n.$$
Employing Theorem \ref{th4.4}, the parts (i) and (ii) can be proved.
\end{proof}

\section{An application to a multiobjective interval-valued noncooperative game}\label{sec5}
\setcounter{equation}{0}
Let us consider a noncooperative game for multiobjective interval-valued functions as the following form
  $$G:=(\Lambda,\{S^i\}_{i\in \Lambda},\{F^i\}_{i\in \Lambda}),$$
 where
\begin{itemize}
\item[(i)] $\Lambda:=\{1,2,\cdots,l\}$ is a finite set of players;
\item[(ii)] for each player $i\in\Lambda$, $S^i$ is a strategy set in $\mathbb{R}^n$, where $S^i:=\{u\in\mathbb{R}^n:g^i_j(u)\leq0,j\in J\}$,
$g^i_j:\mathbb{R}^n\to\mathbb{R}$ is a function for all $j\in J$;
\item[(iii)] $\widehat{S}:=\prod_{i\in\Lambda}S^i$ is the strategy space;
\item[(iv)] for each player $i\in\Lambda$, $F^i(u):=(F^i_1(u),F^i_2(u),\cdots,F^i_m(u))$ is a multiobjective interval-valued loss function on $\widehat{S}$, where $F^i_k(u)=[\underline{F}^i_k(u),\overline{F}^i_k(u)]$ with $F^{ic}_k(u)=\frac{\underline{F}^i_k(u)+\overline{F}^i_k(u)}{2}$ and
    $F^{iw}_k(u)=\frac{\overline{F}^i_k(u)-\underline{F}^i_k(u)}{2}$, $k\in K$.
\end{itemize}
Let $S^{-i}:=\prod_{j\in\Lambda\backslash\{i\}}S^j$.  For each $i\in\Lambda$, we define
$u_{-i}:=(u_1,\cdots,u_{i-1},u_{i+1},\cdots,u_l)\in S^{-i}.$
Moreover, for every $u_i\in S^i$, set $(u_i,u_{-i}):=(u_1,\cdots,u_{i-1},u_i,u_{i+1},\cdots,u_l)\in \widehat{S}$.

\begin{definition}\label{de5.1}
Given $\varepsilon\in\mathbb{R}^m_+$, a strategy profile $\overline{u}=(\overline{u}_1,\overline{u}_2,\cdots,\overline{u}_l)\in \widehat{S}$ is called
\begin{itemize}
\item [(i)] a weak $\varepsilon$-Nash equilibrium (for short, w$\varepsilon$-NE) of $G$ if, for any $i\in\Lambda$, there is no $y_i\in S^i$ such that $$F^i_k(y_i,\overline{u}_{-i})+[0,\varepsilon_k]\prec_{CW} F^i_k(\overline{u}),\;
\forall k\in K;
$$
\item[(ii)] a weak $\varepsilon$-quasi Nash equilibrium (for short, w$\varepsilon$-QNE) of $G$ if, for any $i\in\Lambda$, there is no $y_i\in S^i$ such that $$F^i_k(y_i,\overline{u}_{-i})+[0,\varepsilon_k \|\overline{u}_i-y_i\|]\prec_{CW} F^i_k(\overline{u}),\;
\forall k\in K.
$$
\end{itemize}
\end{definition}

\begin{remark}\label{re5.1}
A point $\overline{u}:=(\overline{u}_1,\overline{u}_2,\cdots,\overline{u}_l)\in \widehat{S}$ is a w$\varepsilon$-QNE (resp., w$\varepsilon$-NE) iff $\overline{u}_i\in S^i$ is a weak $\varepsilon$-quasi-minimal (resp., weak $\varepsilon$-minimal) solution
of the following problem:
$$
\textrm{Min}\; F^i(u_i,\overline{u}_{-i}):=(F^i_1(u_i,\overline{u}_{-i}),F^i_2(u_i,\overline{u}_{-i}),\cdots,F^i_m(u_i,\overline{u}_{-i}))\quad \textrm{s.t.} \;\; u_i\in S^i.
$$
\end{remark}

Now we derive KKT optimality conditions of approximate Nash equilibrium for $G$.
\begin{theorem}\label{th5.1}
Let $\varepsilon\in\mathbb{R}^m_+\backslash\{0\}$ and $\overline{u}:=(\overline{u}_1,\overline{u}_2,\cdots,\overline{u}_l)\in \widehat{S}$ be a w$\varepsilon$-NE of $G$. For each $i\in\Lambda$, $F^i_k(\cdot,\overline{u}_{-i})$ and $g^i_j$ are locally Lipschitz continuous on $\mathbb{R}^n$ for all $k\in K$ and all $j\in J$. Moreover, given $\delta>0$, suppose that BCQ holds on $\overline{B}^i(\overline{u}_i,\delta)$ for all $i\in\Lambda$. Then, for every $i\in\Lambda$, there exists $x_\delta^i\in \overline{B}^i(\overline{u}_i,\delta)$, $\lambda^i:=(\lambda^i_1,\lambda^i_2,\cdots,\lambda^i_m)\in\mathbb{R}^m_+$, $\mu^i:=(\mu^i_1,\mu^i_2,\cdots,\mu^i_p)\in\mathbb{R}^p_+$, with $\sum_{k\in K}\lambda^i_k=1$, such that
\begin{equation}\label{eq5.1}
\begin{split}
0\in&\sum_{k\in K}\lambda^i_k\partial^0 F^i_k(x_\delta^i,\overline{u}_{-i})+\sum_{j\in I(x_\delta^i)}\mu^i_j\partial^0 g^i_j(x_\delta^i)+\frac{1}{\delta}\max_{k\in K}\{\varepsilon_k\}\mathbb{B}^n,\\
&\mu^i_jg^i_j(x_\delta^i)=0, \;\forall j\in J,
\end{split}
\end{equation}
where $\overline{B}^i(\overline{u}_i,\delta):=\{x\in S^i:\|x-\overline{u}_i\|\leq\delta\}$.
\end{theorem}
\begin{proof}
Invoking Remark \ref{re5.1} and Theorem \ref{th4.1}, we can obtain the optimality condition (\ref{eq5.1}).
\end{proof}

\begin{theorem}\label{th5.2}
Let $\varepsilon\in\mathbb{R}^m_+\backslash\{0\}$ and $\overline{u}:=(\overline{u}_1,\overline{u}_2,\cdots,\overline{u}_l)\in \widehat{S}$ be a w$\varepsilon$-QNE of $G$. For each $i\in\Lambda$, $F^i_k(\cdot,\overline{u}_{-i})$ and $g^i_j$ are locally Lipschitz continuous on $\mathbb{R}^n$ for all $k\in K$ and all $j\in J$. Moreover, suppose that BCQ holds at $\overline{u}_i\in S^i$ for all $i\in\Lambda$. Then, for every $i\in\Lambda$, there exists $\lambda^i:=(\lambda^i_1,\lambda^i_2,\cdots,\lambda^i_m)\in\mathbb{R}^m_+$, $\mu^i:=(\mu^i_1,\mu^i_2,\cdots,\mu^i_p)\in\mathbb{R}^p_+$, with $\sum_{k\in K}\lambda^i_k=1$, such that
\begin{equation}\label{eq5.2}
\begin{split}
0\in&\sum_{k\in K}\lambda^i_k\partial^0 F^i_k(\overline{u}_i,\overline{u}_{-i})+\sum_{j\in I(\overline{u}_i)}\mu^i_j\partial^0 g^i_j(\overline{u}_i)+\sum_{k\in K}\lambda^i_k\varepsilon_k \mathbb{B}^n,\\
&\mu^i_jg^i_j(\overline{u}_i)=0, \;\forall j\in J.
\end{split}
\end{equation}
\end{theorem}
\begin{proof}
By Remark \ref{re5.1} and Corollary \ref{co4.1}, we can obtain the optimality condition (\ref{eq5.2}).
\end{proof}

Similar to Definition \ref{de4.2}, we introduce the following notion.
\begin{definition}\label{de5.2}
Let $\overline{u}:=(\overline{u}_1,\overline{u}_2,\cdots,\overline{u}_l)\in \widehat{S}$. For each $i\in\Lambda$, $F^i_k(\cdot,\overline{u}_{-i})$ and $g^i_j$ are locally Lipschitz continuous on $\mathbb{R}^n$ for all $k\in K$ and all $j\in J$. $(F^i(\cdot,\overline{u}_{-i}),g^i_J)$  is called generalized convex at $x_0^i\in S^i$ if for any $x^i\in S^i$, $\widehat{x}^i_k\in\partial^0F^i_k(x^i_0,\overline{u}_{-i})$ for $k\in K$ and $\widehat{z}^i_j\in \partial^0 g_j(x^i_0)$ for $j\in J$, there exists $v^i\in \mathbb{R}^n$ such that
\begin{equation*}
\begin{split}
&\langle \widehat{x}^i_k,v^i\rangle\leq F^{ic}_k(x^i,\overline{u}_{-i})-F^{ic}_k(x^i_0,\overline{u}_{-i})+F^{iw}_k(x^i,\overline{u}_{-i})-F^{iw}_k(x^i_0,\overline{u}_{-i}), \forall k\in K,\\
&\langle \widehat{z}^i_j,v^i\rangle\leq g^i_j(x^i)-g^i_j(x_0^i),\forall j\in J,\\
&\langle y^*,v^i\rangle\leq\|x^i-x_0^i\|,\forall y^*\in\mathbb{B}^n.
\end{split}
\end{equation*}
\end{definition}

\begin{theorem}\label{th5.3}
Let $\varepsilon\in\mathbb{R}^m_+\backslash\{0\}$ and $\overline{u}:=(\overline{u}_1,\overline{u}_2,\cdots,\overline{u}_l)\in \widehat{S}$ satisfy (\ref{eq5.2}), where $F^i_k(\cdot,\overline{u}_{-i})$ and $g^i_j$ are locally Lipschitz continuous on $\mathbb{R}^n$ for all $k\in K$, $i\in\Lambda$ and $j\in J$.
Assume that $(F^i(\cdot,\overline{u}_{-i}),g^i_J)$ is generalized convex at $\overline{u}_i\in S^i$ for all $i\in\Lambda$. Then $\overline{u}$ is a w$\varepsilon$-QNE of $G$.
\end{theorem}
\begin{proof}
In view of Theorem \ref{th4.3} and Remark \ref{re5.1}, we can deduce that $\overline{u}$ is a w$\varepsilon$-QNE of $G$.
\end{proof}

\section{Conclusions}\label{sec6}
\setcounter{equation}{0}
This paper presents some new results for existence theorems and optimality conditions to approximate solutions of MIOPs. Based on interval orders, we obtained the existence of weak $\varepsilon$-minimal solutions of MIOP. Moreover, we  established EVP for multiobjective interval-valued functions. By using this result, we  derived the existence of weak $\varepsilon$-quasi-minimal solutions of MIOP.
We gave the KKT necessary optimality condition for weak $\varepsilon$-minimal and $\varepsilon$-quasi-minimal solutions of MIOP and  obtained the KKT sufficient optimality condition for weak $\varepsilon$-quasi-minimal solutions.
We also devoted a modified $\epsilon$-KKT point of the weak minimal and the $\varepsilon$-quasi-minimal solution of MIOP. Finally, we studied KKT optimality conditions for approximate Nash equilibrium of the game $G$.

Let us end the paper by pointing out some important directions: (i) To consider other types of properties for approximate solutions of MIOP, such as stability, well-posedness, and duality; (ii) To extend our main results to the ones in fuzzy settings; (iii) To study the similar results based on other interval orders.

\section*{Acknowledgements}
The authors are grateful to the editors and referees for their valuable comments and suggestions.

\end{document}